\documentclass[10pt]{amsart}
\usepackage{amssymb}
\usepackage{tikz-cd}
\newtheorem{thm}{Theorem}
\newtheorem*{sthm}{Theorem}
\newtheorem*{llem}{Lemma}
\newtheorem{lem}[thm]{Lemma}
\newtheorem{prop}[thm]{Proposition}

\theoremstyle{definition}
\newtheorem*{rem}{Remark}
\begin{document}
\title[\null]{Chang's Conjectures and Easton collapses}
\author[\null]{Monroe Eskew and Masahiro Shioya}
\address{Kurt G\"odel Research Center,
University of Vienna,
Vienna, Austria.}
\email{monroe.eskew@univie.ac.at}
\address{Institute of Mathematics,
University of Tsukuba,
Tsukuba, 305-8571 Japan.}
\email{shioya@math.tsukuba.ac.jp}
\begin{abstract}
Using Easton collapses,
we give a simplified construction of a model in which 
Chang's Conjecture for triples holds.
\end{abstract}
\thanks{The first author was supported by FWF Project P34603.}
\thanks{The second author was supported by 
JSPS Grant-in-Aid for Scientific Research (C) 21K03334}

\date{}
\maketitle

\section{Introduction}
Around 1960
Chang (see \cite{v}) raised a question about
two-cardinal versions of
the downward L\"owenheim--Skolem theorem.
Let
$\nu>\nu'$ and $\mu>\mu'$ be pairs of cardinals
with $\nu>\mu$ and $\nu'>\mu'$.
By
$(\nu,\nu')\twoheadrightarrow(\mu,\mu')$
we mean the following statement, where
structures are for a countable language
with a unary predicate $\mathtt{P}$:
\begin{quote}
Every structure $\mathcal{N}$ 
of size $\nu$ with $|\mathtt{P}^\mathcal{N}|=\nu'$ has 
an elementary substructure $\mathcal{M}$
of size $\mu$ with $|\mathtt{P}^\mathcal{M}|=\mu'$.
\end{quote}
Chang asked if 
$(\nu,\nu')\twoheadrightarrow(\omega_1,\omega)$ holds in general.
Each instance of 
$(\nu,\nu')\twoheadrightarrow(\mu,\mu')$
is now called Chang's Conjecture (for pairs),
for which we write CC.

CC has been approached from set theory
as a reflection principle.
Kunen~\cite{ku} constructed a model of
$(\omega_2,\omega_1)\twoheadrightarrow(\omega_1,\omega)$,
assuming a huge cardinal exists.
As it turned out, Kunen's method gives a model of
$(\mu^{+2},\mu^+)\twoheadrightarrow(\mu^+,\mu)$,
where $\mu$ is a given regular cardinal.
This contrasts with the constructions due to Silver (see \cite{km})
or Shelah~\cite{s}, which would work at most for 
$(\mu^{+},\mu)\twoheadrightarrow(\omega_1,\omega)$
with $\mu>\omega$ regular 
(see Foreman--Magidor~\cite{fm}).
A survey of Kunen's method can be found in \cite{ff}.

It is natural to consider the case of \emph{two} unary predicates,
say the consistency of
$(\omega_3,\omega_2,\omega_1)\twoheadrightarrow(\omega_2,\omega_1,\omega)$.
Note that this implies that of
$(\omega_3,\omega_2)\twoheadrightarrow(\omega_2,\omega_1)$ 
and thus calls for Kunen's method.
Indeed Foreman~\cite{f} constructed the desired model
by a thorough exploitation of the method:

\begin{sthm}[Foreman]
Suppose $\kappa$ is 2-huge.
Then there is a generic extension in which $\kappa=\omega_1$
and
$(\omega_3,\omega_2,\omega_1)\twoheadrightarrow(\omega_2,\omega_1,\omega)$
holds.
\end{sthm}

The consistency of 
$(\omega_4,\omega_3,\omega_2,\omega_1)\twoheadrightarrow
(\omega_3,\omega_2,\omega_1,\omega)$
remains open.
According to Foreman~\cite[p.~1040]{ff}, 
there is a serious problem of ``ghost coordinates'' 
in extending his construction to the quadruple case.
Further information is not available probably because
the construction in the triple case was already quite complicated.

In this paper 
we give a simplified construction of a model of CC for triples.
The simplification 
allows us to state the problem of ``ghost coordinates''
in extending our construction (see Remark in \S6). 
In a sequel 
we plan to give a model in which
many instances of CC for triples below 
$\omega_\omega$ hold simultaneously.

The novel element of our approach
is an extensive use of projections associated with
term forcing, with which 
Easton collapses fit well (see \cite{sh}).
In contrast, Kunen's method involved the construction of
``universal collapses''  to get complete embeddings.
An analogue of Kunen's theorem (and much more) was proved 
in \cite{eh}:

\begin{sthm}[Eskew]
Suppose $\kappa$ is huge and $j:V\to M$ is a witness.
Let $\mu<\kappa$ be regular.
Then
$E(\mu,\kappa)*\dot{E}(\kappa,j(\kappa))$
forces $\kappa=\mu^+$ 
and
$(\mu^{+2},\mu^+)\twoheadrightarrow(\mu^+,\mu)$.
\end{sthm}

The main result of this paper is
\begin{thm}
Suppose $\kappa$ is 2-huge and $j:V\to M$ is a witness.
Let
$\mu<\kappa$ be regular.
Then 
$$\textstyle\prod\limits^\mathrm{E}_{\xi\in\mathsf{M}\cap\kappa}
E(\xi,\kappa)^{<\xi}\star
\left(\prod\limits^{<\kappa}_{\xi\in\mathsf{M}\cap \kappa}
\dot{E}(\kappa, j(\kappa))^{R(\xi, \kappa)}
\times\!\prod\limits^\mathrm{E}_{\xi\in\mathsf{M}\cap(\kappa,j(\kappa))}
\!\dot{E}(\xi,j(\kappa))\!\right)*
\dot{E}(j(\kappa), j^{2}(\kappa))$$
forces $j^{n}(\kappa)=\mu^{+(n+1)}$ for $n=0,1,2$
and
$(\mu^{+3},\mu^{+2},\mu^+)\twoheadrightarrow(\mu^{+2},\mu^+,\mu)$.
\end{thm}
Here
$\mathsf{M}$ denotes the union of $\{\mu\}$ and the class 
of Mahlo cardinals $>\mu$.
The operation $\star$ and the posets $R(\xi, \kappa)$ are defined in 
\S\S4 and 5 respectively.
The second factor of the middle iterand 
plays a key role to get a master condition in \S6.

\section{Preliminaries}
Our notation is standard.
We refer the reader to \cite{ka}
for the background material.
Throughout the paper
$\mu$ denotes a regular cardinal.

CC can be viewed as an algebraic reflection principle.
Let
$\nu_{0}>\cdots>\nu_{n}$
and
$\mu_{0}>\cdots>\mu_{n}$
be $n+1$-tuples of cardinals (with $n>0$)
such that $\nu_{k}>\mu_{k}$ for every $k\le n$.
Then
$(\nu_{0},\cdots,\nu_{n})\twoheadrightarrow
(\mu_{0},\cdots,\mu_{n})$
iff for every
$f:\nu_{0}^{<\omega}\to\nu_{0}$ 
there is $x\subset\nu_{0}$
closed under $f$
such that $|x\cap\nu_{k}|=\mu_{k}$
for every $k\le n$.

We say that a cardinal $\gamma$ is strongly regular if
$|\gamma^{<\gamma}|= \gamma$.
Forcing with a poset $P$ does not change
the class of strongly regular cardinals $>|P|$.
Let $\mathsf{SR}$ denote the class of strongly regular cardinals (in $V$).

Let
$I$ be a set of ordinals.
Suppose
$Q_\xi$ is a poset for
$\xi\in I$.
For a regular cardinal $\alpha$
let
$\prod^{<\alpha}_{\xi\in I}Q_{\xi}$
denote the $<\!\alpha$-support product.
We write $Q^{<\alpha}$ for
$\prod^{<\alpha}_{\xi<\alpha}Q_{\xi}$
if $Q_{\xi}=Q$ for $\xi<\alpha$.
Similarly
$\prod\nolimits^\mathrm{E}_{\xi\in I}Q_{\xi}$
denotes the Easton support product, where 
a set $d$ of ordinals is Easton if
$\sup(d\cap\gamma)<\gamma$ for every regular $\gamma$.

Let
$S$ be a stationary subset of a regular cardinal
$\gamma>\omega$.
A poset $P$ is $S$-layered if 
$P=\bigcup_{\iota<\gamma}P_{\iota}$
for some increasing sequence $\langle P_{\iota}:{\iota<\gamma}\rangle$ 
of complete suborders of size $<\gamma$ such that
$S\cap C\subset\{\zeta<\gamma:P_{\zeta}=\bigcup_{\iota<\zeta}P_{\iota}\}$
for some club $C\subset \gamma$.
A poset is $\gamma$-cc if
it is $S$-layered for some stationary $S\subset \gamma$
(see \cite{fms}).
The following lemma is standard.

\begin{lem}
Suppose
$\alpha<\gamma$ are regular with $\gamma$ Mahlo.
Let
$S$ be a stationary set of regular cardinals $<\gamma$.
Assume $P_\xi$ is $S$-layered for every $\xi<\gamma$.
Then
$\prod^{<\alpha}_{\xi<\gamma}P_\xi$ and
$\prod^\mathrm{E}_{\xi<\gamma}P_\xi$ are $S$-layered.
\end{lem}

Let $P$ and $R$ be posets.
Suppose
$\pi:P\to R$ is a projection, i.\,e. 
an order-preserving map such that
$\pi(1_P)=1_R$
and if
$r\le_R\pi(p)$, then 
$\pi(p')\le_R r$ for some $p'\le_P p$.
If $D$ is dense open in $R$, then 
$\pi^{-1}[D]$ is dense in $P$.
So if $G\subset P$ is $V$-generic, then
$\pi[G]$ generates a $V$-generic filter 
over $R$, which is denoted by $\pi(G)$.
Moreover $\mathrm{ran}~\pi$ is dense
in $R$ and the map
$r\mapsto\sup\{p\in P:\pi(p)\le r\}$
is a complete embedding of
$R$ into the completion of $P$.
It is easy to see that 
the composite or product of projections is a projection.
The canonical projections as in 
$R*(\dot{Q}\times\dot{Q}')\to R*\dot{Q}\to R$
are denoted by pr.

Suppose
$\dot{Q}$ is an $R$-name for a poset.
Then $T(R,\dot{Q})$ denotes the term forcing, i.\,e. 
the set 
(of representatives under the canonical identification from the class)
$\{\dot{q}\in V^R:R\Vdash\dot{q}\in \dot{Q}\}$
ordered by
$\dot{q}'\le\dot{q}\Leftrightarrow R\Vdash\dot{q}'\ \dot{\le}\ \dot{q}$.
Lemma~5 in \S4 is modeled on

\begin{llem}[Laver]
The identity map
$\mathrm{id}:R\times T(R,\dot{Q})\to R*\dot{Q}$
is a projection.
\end{llem}

Let $j:V\to M$ be an elementary embedding.
Suppose $\varphi:j(P)\to P$ is a projection.
We say that
$p\in j(P)$ is a master condition for 
$\varphi$ if $p'\le j( \varphi(p'))$ for every $p'\le p$ in $j(P)$.
Let
$K\subset j(P)$ be $V$-generic and contain a
master condition for $ \varphi$.
Then
$j[\varphi[K]]\subset K$, so that
$j[\varphi(K)]\subset K$.
Thus $j:V\to M$ extends to
$j:V[\varphi(K)]\to M[K]$ in $V[K]$.
Now let
$\pi:j(R)\to R*\dot{Q}$ be a projection.
Then 
$(1_{j(R)},\dot{q}^{*})$ is a master condition for
$\pi\circ\mathrm{pr}:j(R*\dot{Q})\to R*\dot{Q}$
iff $(\bar{r},\dot{q}^{*})\le j(\pi(\bar{r}))$ 
for every $\bar{r}\in j(R)$.

A cardinal $\kappa$ is $n$-huge if there 
is an elementary embedding $j:V\to M$ 
such that $\kappa$ is the critical point of $j$ and
$M^{j^{n}(\kappa)}\subset M$.
A huge cardinal is a 1-huge cardinal.

Suppose $M$ is an inner model of ZFC.
If $P\in M$ has size $\le\kappa$ in $V$, then
$M^{\kappa}\subset M$ in $V$ implies
$M[G]^{\kappa}\subset M[G]$ in $V[G]$ 
for every $V$-generic $G \subset P$.

\section{The Easton collapse}
In this section we recall basic properties of the Easton collapse.

Suppose $\alpha<\gamma$ are regular cardinals. 
The Easton collapse 
$E(\alpha, \gamma)$ is the poset
$\prod^\mathrm{E}_{\xi\in\mathsf{SR}\cap(\alpha, \gamma)}\xi^{<\alpha}$.
Every linked (i.\,e. pairwise compatible) subset of
$E(\alpha, \gamma)$ of size $<\alpha$ has a lower bound.
Assume further $\gamma$ is Mahlo.
Then
$E(\alpha, \gamma)$ is 
$\gamma$-cc and forces $\gamma = \alpha^{+}$.
In fact it is $S$-layered, where $S$ is the set of
regular cardinals $<\gamma$.

\begin{lem}
Assume $R$ is $\alpha$-cc and of size $\le\alpha$.
\begin{enumerate}
\item If $\dot Q_\xi$ is an $R$-name for a poset for $\xi \in (\alpha,\gamma)$, then 
$T\left(R,\prod^\mathrm{E}_{\xi\in (\alpha,\gamma)} \dot Q_\xi\right) \simeq \prod^\mathrm{E}_{\xi\in (\alpha,\gamma)} T(R,\dot Q_\xi).$
\item For every $\xi \in \mathsf{SR}\cap(\alpha, \gamma)$, $T(R,\dot{\xi^{<\alpha}}) \simeq \xi^{<\alpha}$.

\item Therefore,
$T(R,\dot{E}(\alpha, \gamma))\simeq E(\alpha, \gamma)$.
\end{enumerate}
\end{lem}

See \cite{sh} for the proof.
The dense embedding
$i:q\in E(\alpha, \gamma) \mapsto\dot{q}\in T(R,\dot{E}(\alpha, \gamma))$
is defined by
$${R}\Vdash\dot{q}=\langle \langle
\dot{\eta}({q(\xi)(\zeta)}):\zeta\in\mathrm{dom}~q(\xi)\rangle: 
\xi\in\mathrm{dom}~q\rangle.$$
Here
$\langle\dot{\eta}(\iota): \iota <\gamma\rangle$
is an injective list
of $R$-names for ordinals $<\gamma$ such that
$\langle\dot{\eta}(\iota): \iota <\xi\rangle$ lists
$R$-names for ordinals $<\xi$ for every
$\xi\in\mathsf{SR}\cap(\alpha,\gamma)$.
We say that the dense embedding $i$ and 
the projection 
$\mathrm{id}\times i:
R\times E(\alpha, \gamma) \to R*\dot{E}(\alpha, \gamma)$
are based on the list
$\langle\dot{\eta}(\iota): \iota <\gamma\rangle$.

The following lemma is essentially proved in \cite[Lemma 26]{eh}.

\begin{lem}
Suppose $\kappa$ is huge and $j:V\to M$ is a witness.
Let 
$$\pi:j(R)\to R*\dot{E}(\kappa,j(\kappa))$$
be a projection, where $R$ is $\kappa$-cc and of size $\le\kappa$.
Assume $1_{j(R)}$ 
is a master condition for $\mathrm{pr}\circ\pi:j(R)\to R$.
Then there is a master condition
$(1_{j(R)},\dot{q}^{*})$
for 
$$\pi\circ\mathrm{pr}:j(R*\dot{E}(\kappa,j(\kappa)))\to 
R*\dot{E}(\kappa,j(\kappa)).$$
\end{lem}

\begin{proof}
We claim that 
$${j(R)}\Vdash\dot{q}^{*}=\text{the coordinatewise union of }
\{j(\dot{q}):\exists \bar{r}\in\dot{K}\exists r(\pi(\bar{r})=(r,\dot{q}))\}$$
works, where
$\dot{K}$ is the canonical $j(R)$-name for a generic filter.

To see that
${j(R)}\Vdash\dot{q}^{*}\in
j(\dot{E}(\kappa,j(\kappa)))$, let
$K\subset j(R)$ be $V$-generic.
Work in $V[K]$.
Let
$G*H=\pi(K)\subset R*\dot{E}(\kappa,j(\kappa))$,
which is $V$-generic.
Since $1_{j(R)}$ 
is a master condition for $\mathrm{pr}\circ\pi:j(R)\to R$,
$j:V\to M$ extends to $j:V[G]\to M[K]$.
Note that 
$j[H]\in [M[K]]^{j(\kappa)}\subset M[K]$
because $|j(R)|\le j(\kappa)$ in $V$.
Moreover
$j[H]$ is a directed subset of
$j(\dot{E}(\kappa,j(\kappa)))_{K}=E(j(\kappa),j^{2}(\kappa))^{M[K]}$.

Let
${q}^{*}=\dot{q}^{*}_{K}$, which is the coordinatewise union of $j[H]$.
Then
${q}^{*}\in M[K]$ is a map from
$$\textstyle\bigcup\{j(d):d \subset\mathsf{SR}\cap (\kappa,j(\kappa))
\text{ is Easton in }V[G]\},$$ 
which is an Easton subset of 
$\mathsf{SR}\cap(j(\kappa),j^{2}(\kappa))$.
To see that
${q}^{*}\in E(j(\kappa),j^{2}(\kappa))^{M[K]}$,
let
$\xi\in\mathrm{dom}~{q}^{*}$.
Then
${q}^{*}(\xi)=\bigcup\{j(q)(\xi):q\in H\cap E(\kappa,\delta)^{V[G]}\}$
for some $\delta<j(\kappa)$.
Thus
${q}^{*}(\xi)\in\xi^{<j(\kappa)}$, as desired.

It remains to prove that
$(1_{j(R)},\dot{q}^{*})$
is a master condition for 
$\pi\circ\mathrm{pr}$.
Let
$\bar{r}\in j(R)$ and
$\pi(\bar{r})=(r, \dot{q})$.
It suffices to prove that 
$(\bar{r},\dot{q}^{*})\le (j(r),j(\dot{q}))$
in $j(R*\dot{E}(\kappa,j(\kappa)))$.
By our assumption we have
$\bar{r}\le j(r)$
in $j(R)$.
To see that
$\bar{r}\Vdash\dot{q}^{*}\supset j(\dot{q})$, let
$K\subset {j(R)}$ be $V$-generic with $\bar{r}\in K$.
Then we have
$j(\dot{q})_{K}\subset\dot{q}^{*}_{K}$, 
as desired.
\end{proof}

\section{Iteration with product}
In this section 
we introduce the operation $\star$,
which involves iteration and product.

Let $I$ and $J$ be sets of ordinals.
Suppose
$\pi_\xi:P\to R_\xi$ 
is a projection and
$\dot{Q}_\xi$
is an $R_\xi$-name for a poset
for
$\xi\in I\cup J$.
For a regular cardinal $\alpha$ define
$$P\star\left(\textstyle\prod^{<\alpha}_{\xi\in I}\dot{Q}_\xi
\times\prod\nolimits^\mathrm{E}_{\xi\in J}\dot{Q}_\xi\right)$$
to be the set 
$P\times
\prod\nolimits^{<\alpha}_{\xi\in I}T(R_\xi,\dot{Q}_\xi)
\times
\prod\nolimits^\mathrm{E}_{\xi\in J}T(R_\xi,\dot{Q}_\xi)$
ordered by
\begin{align*}
(p', \dot{q}', \dot{r}')\le(p, \dot{q}, \dot{r})\Leftrightarrow\
& p'\le_Pp,\
\mathrm{dom}~\dot{q}'\supset \mathrm{dom}~\dot{q},\
\mathrm{dom}~\dot{r}'\supset \mathrm{dom}~\dot{r},\\
& \pi_\xi(p')\Vdash_\xi
\dot{q}'(\xi)\ \dot{\le}_\xi\  \dot{q}(\xi)
\text{ for every }\xi\in \mathrm{dom}~\dot{q}\text{ and }\\
& \pi_\xi(p')\Vdash_\xi
\dot{r}'(\xi)\ \dot{\le}_\xi\  \dot{r}(\xi)
\text{ for every }\xi\in \mathrm{dom}~\dot{r}.
\end{align*}
Here
$\Vdash_\xi$ 
denotes the forcing relation associated with 
$R_\xi$.
Note that $\dot{q}$ (say) denotes a sequence of names
rather than a name for a sequence.
The definition is justified by

\begin{lem}
The identity map
$$\mathrm{id}:P\times
\textstyle\prod^{<\alpha}_{\xi\in I}T(R_\xi,\dot{Q}_\xi)
\times
\prod\nolimits^\mathrm{E}_{\xi\in J}T(R_\xi,\dot{Q}_\xi)\to 
P\star\left(\prod\nolimits^{<\alpha}_{\xi\in I}\dot{Q}_\xi
\times\prod\nolimits^\mathrm{E}_{\xi\in J}\dot{Q}_\xi\right)$$
is a projection. 
\end{lem}

\begin{proof}
Suppose
$(p',\dot{q}',\dot{r}')\le(p,\dot{q},\dot{r})$ in
$P\star\left(\prod\nolimits^{<\alpha}_{\xi\in I}\dot{Q}_\xi
\times\prod\nolimits^\mathrm{E}_{\xi\in J}\dot{Q}_\xi\right)$.
It suffices to give
$\dot{q}''\le \dot{q}$ in
$\prod\nolimits^{<\alpha}_{\xi\in I}T(R_\xi,\dot{Q}_\xi)$
and
$\dot{r}''\le \dot{r}$ in
$\prod\nolimits^\mathrm{E}_{\xi\in J}T(R_\xi,\dot{Q}_\xi)$
such that
$$(p',\dot{q}'',\dot{r}'')\le(p',\dot{q}',\dot{r}')\text{ in }
P\star\left(\textstyle\prod^{<\alpha}_{\xi\in I}\dot{Q}_\xi
\times\prod^\mathrm{E}_{\xi\in J}\dot{Q}_\xi\right).$$

Let
$\mathrm{dom}~\dot{q}''=\mathrm{dom}~\dot{q}'$.
For $\xi\in\mathrm{dom}~\dot{q}''$
define
$\dot{q}''(\xi)\in T(R_\xi,\dot{Q}_\xi)$
by
\begin{itemize}
\item
$\pi_\xi(p')\Vdash_\xi \dot{q}''(\xi)= \dot{q}'(\xi)$ and
\item
$r\Vdash_\xi \dot{q}''(\xi)= \dot{q}(\xi)$
for every $r\ \bot_\xi\ \pi_\xi(p')$.
\end{itemize}
(It is understood that
$\Vdash_\xi \dot{q}(\xi)=\dot{1}_\xi$ 
unless $\xi\in \mathrm{dom}~\dot{q}$.)
The definition of $\dot{r}''$ is similar.
It is easy to check that
$\dot{q}''$ and $\dot{r}''$ are as desired.
\end{proof}

We allow $I$ or $J$ to be empty,
in which case we write the poset accordingly.
In particular, if ${\pi}:P\to{R}$ is a projection and 
$\dot{Q}$ is an $R$-name for a poset,
then we define
$P\star\dot{Q}$ suitably and get a projection
${\pi}\times\mathrm{id}:P \star\dot{Q}\to R*\dot{Q}$.

From Lemmas 3 and 5 we get the following result,
which is useful in \S5.
\begin{prop}
Suppose $\alpha<\gamma$ are from $\mathsf{M}-\{\mu\}$.
For $\xi\in\mathsf{M}\cap\alpha$
let
$\pi_{\xi} : P \to R_{\xi}$ be a projection between 
$\alpha$-cc posets of size $\le\alpha$.
Then there is a projection
\begin{align*} 
\mathrm{id} \times \sigma \times \tau: &\ P\times 
\textstyle\prod^{<\alpha}_{\xi\in\mathsf{M}\cap\alpha}{E}(\alpha,\gamma)\times
\textstyle\prod^\mathrm{E}_{\xi\in\mathsf{M}\cap(\alpha,\gamma)}
E(\xi,\gamma)\\
\to &\  P\star
\left(\textstyle\prod^{<\alpha}_{\xi\in\mathsf{M}\cap\alpha}
\dot{E}(\alpha,\gamma)^{R_{\xi}}
\times\textstyle\prod^\mathrm{E}_{\xi\in\mathsf{M}\cap(\alpha,\gamma)}
\dot{E}(\xi,\gamma)\right).
\end{align*} 
\end{prop}

\begin{rem}
Assume as in Theorem~1.
We do not know if 
$$E(\mu,\kappa)*\dot{E}(\kappa,j(\kappa))*\dot{E}(j(\kappa),j^{2}(\kappa))
\Vdash(\mu^{+3},\mu^{+2},\mu^+)\twoheadrightarrow(\mu^{+2},\mu^+,\mu),$$
or if there is a projection 
$$E(\mu,j(\kappa))*\dot{E}(j(\kappa),j^{2}(\kappa)) \to
E(\mu,\kappa)*\dot{E}(\kappa,j(\kappa))*\dot{E}(j(\kappa),j^{2}(\kappa)).$$
We do know, however, that there is a projection 
$$E(\mu,j(\kappa))\star
\dot{E}(j(\kappa),j^{2}(\kappa))^{E(\mu,\kappa)*\dot{E}(\kappa,j(\kappa))}\to
E(\mu,\kappa)*\dot{E}(\kappa,j(\kappa))*\dot{E}(j(\kappa),j^{2}(\kappa)),$$
which motivated the introduction of $\star$.
\end{rem}

\section{The main forcing}
In this section 
we construct posets and projections 
$\pi_{\alpha\gamma}:P(\gamma)\to R(\alpha,\gamma)$ for pairs of
$\alpha<\gamma$ from $\mathsf{M}$.

For $\gamma\in\mathsf{M}-\{\mu\}$ define
$$P(\gamma)=
\textstyle\prod^\mathrm{E}_{\xi\in\mathsf{M}\cap\gamma}
E(\xi,\gamma)^{<\xi}.$$ 
It is easy to see that
$P(\gamma)\subset V_\gamma$ 
is $\mu$-closed.
Since
$P(\gamma)$ has a factor $E(\mu,\gamma)$,
$P(\gamma) \Vdash\gamma=\mu^{+}$
by

\begin{lem}
$P(\gamma)$ is $\gamma$-cc.
\end{lem}

\begin{proof}
Let $S$ be the set of regular cardinals $<\gamma$, 
which is stationary in $\gamma$.
For $\xi\in\mathsf{M}\cap\gamma$,
$E(\xi,\gamma)$ is $S$-layered,
so that $E(\xi,\gamma)^{<\xi}$ is $S$-layered.
Thus
$P(\gamma)$ is $S$-layered, 
which implies the desired result.
\end{proof}

Let
$\langle\dot{\eta}_{\gamma}(\iota): \iota \in\mathrm{Ord}\rangle$
be an injective list of $P(\gamma)$-names for ordinals such that
$\langle\dot{\eta}_{\gamma}(\iota): \iota<\xi\rangle$
lists $P(\gamma)$-names for ordinals $<\xi$
for every $\xi\in\mathsf{SR}-(\gamma+1)$.

By recursion on $\gamma\in\mathsf{M}-\{\mu\}$
we define for each 
$\alpha\in\mathsf{M}\cap \gamma$
\begin{itemize}
\item
a $\gamma$-cc poset $R(\alpha,\gamma)\subset V_\gamma$ 
with a projection
$\pi_{\alpha\gamma}:P(\gamma)\to R(\alpha,\gamma)$ and

\item
an injective list
$\langle\dot{\eta}_{\alpha\gamma}(\iota):\iota \in\mathrm{Ord}\rangle$
of $R(\alpha,\gamma)$-names for ordinals such that
$\langle\dot{\eta}_{\alpha\gamma}(\iota): \iota<\xi\rangle$
lists $R(\alpha,\gamma)$-names for ordinals $<\xi$
for every $\xi\in\mathsf{SR}-(\gamma+1)$.
\end{itemize}
(In practice we need 
$R(\alpha,\gamma)$, $\pi_{\alpha\gamma}$ and 
$\dot{\eta}_{\alpha\gamma}(\iota)$
up to some fixed cardinal,
so that the recursion can be carried out in ZFC.)

First let $R(\mu,\gamma)=P(\gamma)$,
$\pi_{\mu\gamma}=\mathrm{id}$ and
$\dot{\eta}_{\mu\gamma}(\iota)=\dot{\eta}_{\gamma}(\iota)$.
If $\mu<\alpha\in\mathsf{M}\cap\gamma$, define
$$R(\alpha,\gamma)=P(\alpha)\star\left(
\textstyle\prod^{<\alpha}_{\xi\in\mathsf{M}\cap\alpha}
\dot{E}(\alpha,\gamma)^{R(\xi,\alpha)}
\times\textstyle\prod^\mathrm{E}_{\xi\in\mathsf{M}\cap(\alpha,\gamma)}
\dot{E}(\xi,\gamma)\right)$$
and
$\pi_{\alpha\gamma}:P(\gamma)\to R(\alpha,\gamma)$
by composing the following projections:
\begin{align*} 
P(\gamma)
\simeq &\,  \textstyle\prod^\mathrm{E}_{\xi\in\mathsf{M}\cap\alpha}
E(\xi,\gamma)^{<\xi}\times E(\alpha,\gamma)^{<\alpha} \times
\textstyle\prod^\mathrm{E}_{\xi\in\mathsf{M}\cap(\alpha,\gamma)}
E(\xi,\gamma)^{<\xi}\\
\to &\,  \textstyle\prod^\mathrm{E}_{\xi\in\mathsf{M} \cap\alpha}
E(\xi,\alpha)^{<\xi}\times 
\textstyle\prod^{<\alpha}_{\xi\in\mathsf{M}\cap\alpha}{E}(\alpha,\gamma)\times
\textstyle\prod^\mathrm{E}_{\xi\in\mathsf{M}\cap(\alpha,\gamma)}
E(\xi,\gamma)\\
\to &\, P(\alpha)\star
\left(\textstyle\prod^{<\alpha}_{\xi\in\mathsf{M}\cap\alpha}
\dot{E}(\alpha,\gamma)^{R(\xi,\alpha)}
\times\textstyle\prod^\mathrm{E}_{\xi\in\mathsf{M}\cap(\alpha,\gamma)}
\dot{E}(\xi,\gamma)\right)
=R(\alpha,\gamma).
\end{align*} 
The first projection is defined by
$$(p,q,r) \mapsto 
(\langle\langle  p(\xi)(\zeta)|\alpha:\zeta\in\mathrm{dom}~p(\xi)\rangle:
\xi\in\mathrm{dom}~p\rangle, q|\mathsf{M}, 
\langle r(\xi)(\alpha):\xi\in\mathrm{dom}~r\rangle).$$
Note the use of the components indexed by $\alpha$ in the third coordinate.
This will be key for obtaining the commuting diagram in Lemma~9.
We get the second projection by Proposition~6, based on the lists of names  $\langle\dot\eta_{\alpha}(\iota) : \iota < \gamma \rangle$ and $\langle\dot\eta_{\xi\alpha}(\iota) : \iota < \gamma \rangle$.
To define
$\langle\dot{\eta}_{\alpha\gamma}({\iota}):\iota\in\mathrm{Ord}\rangle$,
we identify
$R(\alpha,\gamma)$ with an iteration 
$P(\alpha)*\dot{Q}_{\alpha\gamma}$.
Let
$$R(\alpha,\gamma)\Vdash\dot{\eta}_{\alpha\gamma}(\iota)=
\dot{\rho}_{\alpha\gamma}(\dot{\eta}_{\alpha}(\iota)).$$
Here 
$\langle\dot{\rho}_{\alpha\gamma}({\iota}):\iota\in\mathrm{Ord}\rangle$ 
is forced to be an injective list of $\dot{Q}_{\alpha\gamma}$-names for ordinals
such that
$\langle\dot{\rho}_{\alpha\gamma}({\iota}):\iota<\xi\rangle$
lists $\dot{Q}_{\alpha\gamma}$-names for ordinals $<\xi$
for every $\xi\in\mathsf{SR}-(\gamma+1)$.

Suppose $\alpha<\gamma$ are from $\mathsf{M}-\{\mu\}$.
Define (the latter as a set)
\begin{align*}
{Q}(\alpha,\gamma)
&=\textstyle\prod\nolimits^{<\alpha}_{\xi\in\mathsf{M}\cap\alpha}{E}(\alpha,\gamma)
\times
\prod\nolimits^\mathrm{E}_{\xi\in\mathsf{M}\cap(\alpha,\gamma)}E(\xi,\gamma),\\
\dot{Q}(\alpha,\gamma)
&=
\textstyle\prod\nolimits^{<\alpha}_{\xi\in\mathsf{M} \cap\alpha}
T(R(\xi,\alpha),\dot{E}(\alpha,\gamma))\times
\prod\nolimits^\mathrm{E}_{\xi\in\mathsf{M}\cap(\alpha,\gamma)}
T(P(\alpha),\dot{E}(\xi,\gamma)).
\end{align*}
We write 
$P(\alpha)\star\dot{Q}(\alpha,\gamma)$ for $R(\alpha,\gamma)$.
Thus
$\pi_{\alpha\gamma}$ is the composite of the projections
$${P}(\gamma)\xrightarrow{\pi^{\times}_{\alpha\gamma}}P(\alpha)\times{Q}(\alpha,\gamma)
\xrightarrow{\mathrm{id}\times i _{\alpha\gamma}}P(\alpha)\star\dot{Q}(\alpha,\gamma).$$

Since ${P}(\alpha)$ is $\alpha$-cc and
${Q}(\alpha,\gamma)$ is $\alpha$-closed,
${P}(\alpha)$ forces ${Q}(\alpha,\gamma)$ to be $\alpha$-Baire.
Thus
${P}(\alpha)\times{Q}(\alpha,\gamma)\Vdash\gamma=\alpha^{+}=\mu^{+2}$
because it is $\gamma$-cc and has a factor ${E}(\alpha,\gamma)$.
Therefore
$P(\alpha)\star\dot{Q}(\alpha,\gamma)\Vdash\gamma=\alpha^{+}=\mu^{+2}$
because of the projections
$$P(\alpha)\times{Q}(\alpha,\gamma)
\to P(\alpha)\star\dot{Q}(\alpha,\gamma)
\to P(\alpha)\star\dot{E}(\alpha,\gamma)^{{R}(\mu,\alpha)}
= P(\alpha)*\dot{E}(\alpha,\gamma).$$

Assume as in Theorem~1.
The poset there can be written as
$$R(\kappa,j(\kappa))*\dot{E}(j(\kappa),j^{2}(\kappa))
=P(\kappa)\star \dot{Q}(\kappa,j(\kappa))*\dot{E}(j(\kappa),j^{2}(\kappa)).$$
To prove Lemma~9 in \S6,
we need to add one requirement on our construction.
Let
$\mu<\alpha\in\mathsf{M}\cap \kappa$.
Then 
$j(P(\kappa))=P(j(\kappa))$ and
$j(P(\alpha)\star\dot{Q}(\alpha,\kappa))=
P(\alpha)\star\dot{Q}(\alpha,j(\kappa))$
by $M^{j(\kappa)}\subset M$.
The additional requirement is
$$j(\dot{\eta}_{\alpha \kappa}| \kappa)
=\dot{\eta}_{\alpha j(\kappa)}|j(\kappa),$$ 
which implies
$j(i_{\alpha \kappa})=i_{\alpha j(\kappa)}$
and in turn
$j(\pi_{\alpha \kappa})=\pi_{\alpha j(\kappa)}$.
This can be achieved as follows.  At the start, fix some well-ordering $\lhd$ of $V_\kappa$, and define the lists of names for ordinals below $\kappa$ using $\lhd$ and its canonical extensions to $V_\kappa^P$ for $P \in V_\kappa$.  Then carry out the construction up to $j(\kappa)$ using $j(\lhd)$.

\section{Master conditions}
Throughout this section we work under the hypothesis of Theorem~1
unless otherwise stated.
The main result of this section is

\begin{prop}
There is a master condition
$(1_{j(P(\kappa))},\dot{s}^{*})$ 
for 
$$\pi_{\kappa j(\kappa)}\circ\mathrm{pr}:
j(P(\kappa)\star\dot{Q}(\kappa,j(\kappa)))\rightarrow
P(\kappa)\star\dot{Q}(\kappa,j(\kappa)).$$
\end{prop}

The component $\dot{s}^{*}$ will have the form 
$(\langle\dot{q}^{*}_{\alpha}:\alpha\in\mathsf{M}\cap\kappa\rangle, \dot{r}^{*})$, 
where
$(1_{j(P(\kappa))},\dot{q}^{*}_{\alpha})$ 
is a master condition for a projection
$$j(P(\kappa)\star
\dot{E}(\kappa,j(\kappa))^{R(\alpha,\kappa)})
\rightarrow
P(\kappa)\star\dot{E}(\kappa,j(\kappa))^{R(\alpha,\kappa)}.$$
To get $\dot{q}^{*}_{\alpha}$
for $\alpha >\mu$, 
we extract the following lemma from \cite[Lemma~3.1]{f},
for which we assume $j:V\to M$ to be elementary.

\begin{llem}[Foreman]
Let
$\pi:P\to R$ be a projection and
$\dot{Q}$ an $R$-name for a poset.
Suppose
$\sigma:j(P)\to P\star\dot{Q}$ and
$\tau:j(R)\to R*\dot{Q}$ are projections.
Assume
\begin{itemize}
\item[(1)]
$j(P\star\dot{Q})=j(P)\star j(\dot{Q})$ and
$j(R*\dot{Q})=j(R)* j(\dot{Q})$,

\item[(2)]
$1_{j(P)}$ 
is a master condition for 
$\mathrm{pr}\circ\sigma:j(P)\rightarrow P$,

\item[(3)]
$(1_{j(R)},\dot{q}^*)$
is a master condition for 
$\tau\circ\mathrm{pr}:j(R*\dot{Q})\rightarrow R*\dot{Q}$
and

\item[(4)]
the following diagram of projections commutes:
\end{itemize}
\begin{center}
\begin{tikzpicture}[auto]
\node (d) at (2, 1.5) {$P\star\dot{Q}$};
\node (a) at (0, 0) {$j(R)$};   
\node (b) at (0, 1.5) {$j(P)$};
\node (c) at (2, 0) {$R*\dot{Q}$};

\draw[->] (a) to node[swap] {$\scriptstyle\tau$} (c);
\draw[->] (d) to node {$\scriptstyle\pi\times \mathrm{id}$} (c);
\draw[->] (b) to node {$\scriptstyle\sigma$} (d);
\draw[->] (b) to node[swap] {$\scriptstyle j(\pi)$} (a);
\end{tikzpicture}
\end{center}
Then
$(1_{j(P)},\dot{q}^*)$ is a master condition for 
$\sigma\circ\mathrm{pr}:j(P\star\dot{Q})\rightarrow P\star\dot{Q}$.
\end{llem}

\begin{proof}
Let
$\bar{p}\in j(P)$ and
$\sigma(\bar{p})=(p,\dot{q})$.
It suffices to prove (as in the case of $*$) that
$(\bar{p},\dot{q}^*)\le(j(p),j(\dot{q}))$
in $j(P)\star j(\dot{Q})$.
By (2)
we have $\bar{p}\le j(p)$ in $j(P)$.
It remains to prove that
$j(\pi)(\bar{p})\Vdash\dot{q}^*\ \dot{\le}\ j(\dot{q})$.
By (3) and (4) we have
\begin{align*}
(j(\pi)(\bar{p}),\dot{q}^*)
&\le j(\tau(j(\pi)(\bar{p})))\\
&=j((\pi\times\mathrm{id})(\sigma(\bar{p})))\\
&=j((\pi\times\mathrm{id})(p,\dot{q}))\\
&=(j(\pi(p)),j(\dot{q})),
\end{align*}
which implies the desired result.
\end{proof}

While the conditions (1) and (2) are easy to verify,
(3) and (4) require the combination of Lemma~4 and

\begin{lem}
Suppose $\mu<\alpha\in\mathsf{M}\cap\kappa$.
Then there is a projection $\tau_{\alpha}$ such that 
\begin{center}
\begin{tikzpicture}[auto]
\node (d) at (5, 1.5) {$P(\kappa)\star\dot{E}(\kappa,j(\kappa))^{R(\alpha,\kappa)}$};
\node (a) at (0, 0) {$j(R(\alpha,\kappa))$};   
\node (b) at (0, 1.5) {$j(P(\kappa))$};
\node (c) at (5, 0) {$R(\alpha,\kappa)*\dot{E}(\kappa,j(\kappa))$};

\draw[->] (a) to node[swap] {$\scriptstyle\tau_{\alpha}$} (c);
\draw[->] (d) to node {$\scriptstyle\pi_{\alpha\kappa}\times \mathrm{id}$} (c);
\draw[->] (b) to node {$\scriptstyle(\mathrm{id}\times\mathrm{pr}_{\alpha})\circ\pi_{\kappa j(\kappa)}$} (d);
\draw[->] (b) to node[swap] {$\scriptstyle j(\pi_{\alpha\kappa})$} (a);
\end{tikzpicture}
\end{center}
commutes and
$1_{j(R(\alpha,\kappa))}$ 
is a master condition for
$\mathrm{pr}\circ\tau_{\alpha}:j(R(\alpha,\kappa))\to
R(\alpha,\kappa)$.
\end{lem}

Here 
$\mathrm{pr}_{\alpha}:\dot{Q}(\kappa,j(\kappa))\to
T(R(\alpha,\kappa),\dot{E}(\kappa,j(\kappa))$
is defined by 
$(\dot{q}, \dot{r})\mapsto \dot{q}(\alpha)$.
Recall that
$j(\pi_{\alpha \kappa})$ is equal to
$\pi_{\alpha j(\kappa)}:P(j(\kappa))\to P(\alpha)\star\dot{Q}(\alpha,j(\kappa))$
by our construction.

\begin{proof}
It is straightforward to check that the following diagram commutes:
\begin{center}
\begin{tikzpicture}[auto]
\node (h) at (4, 3) {$P(j(\kappa))$}; 
\node (i) at (8, 3) {$P(\kappa)\times{Q}(\kappa,j(\kappa))$};
\node (f) at (8, 2) {$P(\kappa)\star\dot{Q}(\kappa,j(\kappa))$};
\node (g) at (0, 3) {$P(\alpha)\times{Q}(\alpha,j(\kappa))$}; 
\node (d) at (4, 1.5) {$P(\alpha)\times{Q}(\alpha,\kappa)\times{E}(\kappa,j(\kappa))$};
\node (c) at (8, 1) {$P(\kappa)\star\dot{E}(\kappa,j(\kappa))^{P(\alpha)\star\dot{Q}(\alpha,\kappa)}$};
\node (e) at (0, 2) {$P(\alpha)\star\dot{Q}(\alpha,j(\kappa))$};   
\node (a) at (0, 0) {$P(\alpha)\star(\dot{Q}(\alpha,\kappa)\times 
\dot{E}(\kappa,j(\kappa)))$};
\node (b) at (8, 0) {$P(\alpha)\star\dot{Q}(\alpha,\kappa)*\dot{E}(\kappa,j(\kappa))$};

\draw[->] (c) to node {$\scriptstyle\pi_{\alpha\kappa}\times\mathrm{id}$} (b);
\draw[->] (d) to node[swap] {$\scriptstyle\mathrm{id}\times i_{\alpha\kappa}\times i ^{\mu\alpha}$} (a);
\draw[->] (e) to node[swap] {$\scriptstyle\mathrm{id}\times\mathrm{res}_\kappa$} (a);
\draw[->] (d) to node[swap] {$\scriptstyle\mathrm{id}\times i _{\alpha\kappa}\times i ^{\alpha\kappa}$} (b);
\draw[->] (f) to node {$\scriptstyle\mathrm{id}\times\mathrm{pr}_{\alpha}$} (c);
\draw[->] (h) to node[swap] {$\scriptstyle\pi^{\times}_{\alpha j(\kappa)}$} (g);
\draw[->] (g) to node[swap] {$\scriptstyle\mathrm{id}\times i_{\alpha j(\kappa)}$} (e);
\draw[->] (g) to node {$\scriptstyle\mathrm{id}\times\mathrm{res}^\times_{\kappa}$} (d);
\draw[->] (i) to node[swap] {$\scriptstyle\pi^{\times}_{\alpha\kappa}\times\mathrm{pr}^\times_{\alpha}$} (d);
\draw[->] (h) to node {$\scriptstyle\pi^{\times}_{\kappa j(\kappa)}$} (i);
\draw[->] (i) to node {$\scriptstyle\mathrm{id}\times i_{\kappa j(\kappa)}$} (f);
\end{tikzpicture}
\end{center}
Here 
$\mathrm{res}^\times_{\kappa}:({q}, {r})\mapsto
((\langle {q}(\xi)|\kappa:\xi\in\mathrm{dom}~{q}\rangle,
\langle {r}(\xi)|\kappa:\xi\in\mathrm{dom}~{r}\cap\kappa\rangle),
{r}(\kappa))$ 
and
$\mathrm{pr}^\times_{\alpha}:({q},{r})\mapsto{q}(\alpha)$,
so that the projections
$$P(j(\kappa))\to
P(\alpha)\times{Q}(\alpha,\kappa)\times{E}(\kappa,j(\kappa))$$
map $p$ to the triple of
\begin{itemize}
\item
$\langle\langle p(\xi)(\zeta)|\alpha:\zeta\in\mathrm{dom}~p(\xi)\rangle:
\xi\in\mathrm{dom}~p\cap\alpha\rangle$,
\item
$(\langle p(\alpha)(\zeta)|\kappa:\zeta\in\mathsf{M}\cap\mathrm{dom}~p(\alpha)\rangle,
\langle p(\xi)(\alpha)|\kappa:\xi\in\mathrm{dom}~p\cap(\alpha,\kappa)\rangle)$
and
\item
$p(\kappa)(\alpha)$.
\end{itemize}
In addition
$\mathrm{res}_\kappa:
(\dot{q}, \dot{r})\mapsto
((\langle \dot{q}(\xi)|\kappa:\xi\in\mathrm{dom}~\dot{q}\rangle,
\langle \dot{r}(\xi)|\kappa:\xi\in\mathrm{dom}~\dot{r}\cap\kappa\rangle),
\dot{r}(\kappa))$,
and 
$i ^{\beta \gamma}:{E}(\kappa,j(\kappa))\to
T(R(\beta, \gamma),\dot{E}(\kappa,j(\kappa)))$
is defined as in Lemma~3 for
$(\beta, \gamma)=(\mu, \alpha)$, $(\alpha, \kappa)$.

It remains to construct a map 
$$\varphi:T(P(\alpha),\dot{E}(\kappa,j(\kappa)))\to
T(P(\alpha)\star\dot{Q}(\alpha,\kappa),\dot{E}(\kappa,j(\kappa)))$$
such that the following diagram of projections commutes:
\begin{center}
\begin{tikzpicture}[auto]
\node (c) at (4, 1.5) {$P(\alpha)\times{Q}(\alpha,\kappa)\times{E}(\kappa,j(\kappa))$}; 
\node (a) at (0, 0) {$P(\alpha)\star(\dot{Q}(\alpha,\kappa)\times \dot{E}(\kappa,j(\kappa)))$};   
\node (b) at (8, 0) {$P(\alpha)\star\dot{Q}(\alpha,\kappa)*\dot{E}(\kappa,j(\kappa))$};

\draw[->] (a) to node[swap] {$\scriptstyle\mathrm{id}\times\mathrm{id}\times\varphi$} (b);
\draw[->] (c) to node[swap] {$\scriptstyle\mathrm{id}\times i_{\alpha\kappa}\times i^{\mu\alpha}$} (a);
\draw[->] (c) to node {$\scriptstyle\mathrm{id}\times i _{\alpha\kappa}\times i^{\alpha\kappa}$} (b);
\end{tikzpicture}
\end{center}
Note that the resulting projection
$$\mathrm{pr}\circ\tau_{\alpha}:
j(P(\alpha)\star\dot{Q}(\alpha,\kappa))\to
P(\alpha)\star\dot{Q}(\alpha,\kappa)$$
does not depend on the choice of $\varphi$, for which
$1_{j(P(\alpha)\star\dot{Q}(\alpha,\kappa))}$ 
is easily seen to be a master condition.

To define 
$\varphi$,
first recall that 
$i^{\mu\alpha}$ and $i^{\alpha\kappa}$
are based on
$\langle\dot{\eta}_{\alpha}(\iota):\iota<j(\kappa)\rangle$
and
$\langle\dot{\eta}_{\alpha\kappa}(\iota):\iota<j(\kappa)\rangle$
respectively.
Let $r\in{E}(\kappa,j(\kappa))$.
Then
$i^{\mu\alpha}:r\mapsto\dot{s}$ and 
$i^{\alpha\kappa}:r\mapsto\dot{t}$ are defined by
\begin{align}
P(\alpha) & \Vdash\dot{s}
=\langle \langle
\dot{\eta}_{\alpha}(r(\xi)(\zeta)):\zeta\in\mathrm{dom}~r(\xi)\rangle: 
\xi\in\mathrm{dom}~r\rangle,\\
P(\alpha)\star\dot{Q}(\alpha,\kappa) &
\Vdash\dot{t}=
\langle \langle
\dot{\eta}_{\alpha\kappa}(r(\xi)(\zeta)):\zeta\in\mathrm{dom}~r(\xi)\rangle: 
\xi\in\mathrm{dom}~r\rangle.
\end{align}

Let $\dot{s}\in T(P(\alpha),\dot{E}(\kappa,j(\kappa)))$.
We may assume (1) for some $r\in {E}(\kappa,j(\kappa))$.
Define
$\varphi:\dot{s}\mapsto\dot{t}\in 
T(P(\alpha)\star\dot{Q}(\alpha,\kappa),\dot{E}(\kappa,j(\kappa)))$
by (2).
Then it is easy to see that this makes the diagram commute.

It remains to prove that
$$\mathrm{id}\times\mathrm{id}\times\varphi:
P(\alpha)*(\dot{Q}_{\alpha\kappa}\times \dot{E}(\kappa,j(\kappa)))
\to P(\alpha)*\dot{Q}_{\alpha\kappa}*\dot{E}(\kappa,j(\kappa))$$
is a projection, where
$P(\alpha)*\dot{Q}_{\alpha\kappa}\simeq P(\alpha)\star\dot{Q}(\alpha,\kappa)$.
This is because
$$\mathrm{id}\times\varphi:
\dot{Q}_{\alpha\kappa}\times \dot{E}(\kappa,j(\kappa))
\to\dot{Q}_{\alpha\kappa}*\dot{E}(\kappa,j(\kappa))$$
is forced to be based on the list
$\langle\dot{\rho}_{\alpha\kappa}({\iota}):\iota<j(\kappa)\rangle$
of $\dot{Q}_{\alpha\kappa}$-names for ordinals.
\end{proof}

It remains to give
$\dot{r}^{*}$, for which we prove

\begin{lem}
Suppose
$\pi:j(P)\to 
P\star\prod^\mathrm{E}_{\xi\in\mathsf{M}\cap(\kappa,j(\kappa))}
\dot{E}(\xi,j(\kappa))$
is a projection, where
$P$ is $\kappa$-cc and of size $\le\kappa$.
Assume $1_{j(P)}$ 
is a master condition for 
$\mathrm{pr}\circ{\pi}:j(P)\rightarrow P$.
Then there is a master condition
$(1_{j(P)}, \dot{r}^{*})$ for 
$${\pi}\circ\mathrm{pr}:
j\left(P\star
\textstyle\prod^\mathrm{E}_{\xi\in\mathsf{M} \cap(\kappa,j(\kappa))}\dot{E}(\xi,j(\kappa))\right)
\rightarrow
P\star\textstyle\prod^\mathrm{E}_{\xi\in\mathsf{M} \cap(\kappa,j(\kappa))}\dot{E}(\xi,j(\kappa)).$$
\end{lem}

Note that
$j\left(P\star
\textstyle\prod^\mathrm{E}_{\xi\in\mathsf{M} \cap(\kappa,j(\kappa))}\dot{E}(\xi,j(\kappa))\right)
=j(P)\star
\textstyle\prod^\mathrm{E}_{\xi\in\mathsf{M}\cap(j(\kappa),j^{2}(\kappa))}
\dot{E}(\xi,j^{2}(\kappa))$.

\begin{proof}
Note that by Lemma~3, 
$$\textstyle\prod^{\mathrm{E}}_{\xi\in\mathsf{M}\cap(\kappa,j(\kappa))} T\left(P,\dot E(\xi,j(\kappa))\right) \simeq T \left( P, \textstyle\prod^\mathrm{E}_{\xi\in\mathsf{M} \cap(\kappa,j(\kappa))}\dot{E}(\xi,j(\kappa))\right).$$
Thus there is a dense embedding $d : P\star\prod^\mathrm{E}_{\xi\in\mathsf{M}\cap(\kappa,j(\kappa))}
\dot{E}(\xi,j(\kappa)) \to P * \dot Q$, where $\dot Q$ is a $P$-name for a $\kappa^+$-directed-closed poset of size $j(\kappa)$, such that $d$ is the identity on the first coordinate.

Let $K \subset j(P)$ be $V$-generic, and let $G * H = d \circ \pi(K)$.  In $M[K]$, 
$j[H]$ is a directed subset of size $j(\kappa)$ in the $j(\kappa)^+$-directed-closed poset $j(\dot Q)_K$.  We can thus take a $j(P)$-name $\dot q^*$ for a lower bound in $j(\dot Q)$.  Translating $\dot q^*$ into a condition $\dot r^* \in \textstyle\prod^\mathrm{E}_{\xi\in\mathsf{M}\cap(j(\kappa),j^{2}(\kappa))}
\dot{E}(\xi,j^{2}(\kappa))$, it is straightforward to check that $r^*$ is as desired.
\end{proof}

\begin{proof}[Proof of Proposition~8]
Let $\alpha\in\mathsf{M}\cap\kappa$.
We claim that there is a master condition 
$(1_{j(P(\kappa))}, \dot{q}^{*}_{\alpha})$ 
for 
$$\sigma\circ\mathrm{pr}:
j(P(\kappa)\star\dot{E}(\kappa,j(\kappa))^{R(\alpha, \kappa)})
\to P(\kappa)\star\dot{E}(\kappa,j(\kappa))^{R(\alpha, \kappa)},$$
where $\sigma$ is
$(\mathrm{id}\times\mathrm{pr}_{\alpha}) \circ \pi_{\kappa j(\kappa)}:
j(P(\kappa))
\to P(\kappa)\star\dot{E}(\kappa,j(\kappa))^{R(\alpha, \kappa)}$.

If $\alpha=\mu$,
then 
$R(\alpha,\kappa)=P(\kappa)$, so that
the claim follows from Lemma~4 with $\pi=\sigma$.
Next assume $\alpha>\mu$.
Then there is a projection
$$\tau_{\alpha}:
j({R(\alpha, \kappa)})\to
{R(\alpha, \kappa)}*\dot{E}(\kappa,j(\kappa))$$
as in Lemma~9.
By Lemma~4 with $\pi=\tau_{\alpha}$
there is a master condition 
$(1_{j\left(R(\alpha,\kappa)\right)}, \dot{q}^{*}_{\alpha})$
for 
$$\tau_{\alpha}\circ\mathrm{pr}:
j({R(\alpha, \kappa)}*\dot{E}(\kappa,j(\kappa)))\to
{R(\alpha, \kappa)}*\dot{E}(\kappa,j(\kappa)).$$
The claim follows from Foreman's lemma with
$\pi=\pi_{\alpha \kappa}$, 
$\sigma=(\mathrm{id}\times\mathrm{pr}_{\alpha}) \circ \pi_{\kappa j(\kappa)}$
and $\tau=\tau_{\alpha}$.

By Lemma~10 we get a master condition
$(1_{j(P(\kappa))},\dot{r}^{*})$
for 
$$\pi\circ\mathrm{pr}:
j\left(P(\kappa)\star\textstyle\prod^\mathrm{E}_
{\xi\in\mathsf{M}\cap(\kappa,j(\kappa))}\dot{E}(\xi, j(\kappa))\right)
\to
P(\kappa)\star\textstyle\prod^\mathrm{E}_{\xi\in\mathsf{M}\cap(\kappa,j(\kappa))}
\dot{E}(\xi, j(\kappa)),$$
where
$\pi$ is 
$(\mathrm{id}\times\mathrm{pr})\circ\pi_{\kappa j(\kappa)}:
j(P(\kappa))\to
P(\kappa)\star\textstyle\prod^\mathrm{E}_{\xi\in\mathsf{M}\cap(\kappa,j(\kappa))}
\dot{E}(\xi, j(\kappa))$.

Now it is easy to see that
$(1_{j(P(\kappa))},
({\langle \dot{q}^{*}_{\alpha}:\alpha\in\mathsf{M}\cap\kappa\rangle}, \dot{r}^{*}))$
is as desired.
\end{proof}

\begin{rem}
Suppose $\kappa$ is 3-huge and $j:V\to M$ is a witness.
Toward a model of
$(\omega_4,\omega_3,\omega_2,\omega_1)\twoheadrightarrow
(\omega_3,\omega_2,\omega_1,\omega)$,
it seems natural to force with a poset of the form
$$P(\kappa)\star\dot{Q}(\kappa,j(\kappa))
\star\dot{R}(j(\kappa),j^{2}(\kappa))*\dot{E}(j^{2}(\kappa),j^{3}(\kappa)).$$
Here
$\dot{Q}(\alpha,\gamma)$ and $\dot{R}(\alpha,\gamma)$ 
are defined by recursion for a pair of
(say) Mahlo cardinals $\alpha<\gamma$ so that
there is a projection
\begin{align*}
P(j(\kappa))&\star\dot{Q}(j(\kappa),j^{2}(\kappa))
\star\dot{R}(j^{2}(\kappa),j^{3}(\kappa))\\
\to
P(\kappa)\star\dot{Q}(\kappa,j(\kappa))
&\star\dot{R}(j(\kappa),j^{2}(\kappa))*\dot{E}(j^{2}(\kappa),j^{3}(\kappa)).
\end{align*}
To meet the requirement, 
$\dot{R}(j^{2}(\kappa),j^{3}(\kappa))$
should have a factor
$$\dot{E}(j^{2}(\kappa),j^{3}(\kappa))
^{P(\kappa)\star\dot{Q}(\kappa,j(\kappa))
\star\dot{R}(j(\kappa),j^{2}(\kappa))}.$$
Then
$\dot{R}(j(\kappa),j^{2}(\kappa))$
would have a factor  
$$\textstyle\prod_{\alpha\in\mathsf{M}\cap \kappa}
\dot{E}(j(\kappa),j^{2}(\kappa))^{P(\alpha)\star\dot{Q}(\alpha, \kappa)
\star\dot{R}(\kappa,j(\kappa))}.$$
If the structure of the forcing is to be broadly similar to ours,
the stage $\dot R(j(\kappa),j^2(\kappa))$ should not add bounded subsets of $j(\kappa)$.
Therefore, the above product must be taken with full support.

In order to construct a master condition, we need $\dot r^* \in \dot R(j^2(\kappa),j^3(\kappa))$ 
that is forced to be below $j(r)$ for all $r \in K \subset \dot R(j(\kappa),j^2(\kappa))$, where $K$ is the generic filter.
For $\gamma \in \mathsf M \cap j(\kappa)$, we would like to argue that 
$$\dot r^*(\gamma) \in \dot{E}(j^{2}(\kappa),j^{3}(\kappa))
^{P(\gamma)\star\dot{Q}(\gamma,j(\kappa))
\star\dot{R}(j(\kappa),j^{2}(\kappa))}$$
can be taken as the coordinatewise union of $\{ j(r)(\gamma) : r \in K \}$.
This set would need to be computable in the intermediate extension 
$$W = V^{P(\gamma) \star \dot Q(\gamma,j(\kappa)) \star \dot R(j(\kappa),j^2(\kappa))} \subset V^{P(j(\kappa)) \star \dot Q(j(\kappa),j^2(\kappa))},$$
in which the relevant Easton collapse is $j^2(\kappa)$-closed.
But for $\gamma \in j(\kappa) \setminus \kappa$, 
we would seem to need the entire filter $K$ in order to compute $\{ j(r)(\gamma) : r \in K \}$.  This would only be available in an extension 
$W' \supseteq  V^{P(\kappa) \star \dot Q(\kappa,j(\kappa)) \star \dot R(j(\kappa),j^2(\kappa))}$.
But in $W'$, $\gamma$ is not a cardinal, and thus $K$ is not available to $W$.

This issue is avoided in our construction.  Since our $\dot Q(\kappa,j(\kappa))$ uses ${<}\kappa$ supports in its first coordinate, we can take master conditions at each coordinate and concatenate them.
\end{rem}

\section{Proof of Theorem~1}
We are ready for
\begin{proof}[Proof of Theorem~1]
Since
${R}(\kappa,j(\kappa))=P(\kappa)\star\dot{Q}(\kappa,j(\kappa))$
forces $j(\kappa)=\kappa^{+}=\mu^{+2}$ and
$j^{2}(\kappa)$ is Mahlo,
${R}(\kappa,j(\kappa))*\dot{E}(j(\kappa),j^{2}(\kappa))
\Vdash j^{2}(\kappa)=j(\kappa)^{+}=\kappa^{+2}=\mu^{+3}$.

It remains to prove that
$(j^{2}(\kappa),j(\kappa),\kappa)\twoheadrightarrow(j(\kappa),\kappa,\mu)$
is forced as well.
By Proposition~8 there is a master condition 
$(1_{j(P(\kappa))}, \dot{s}^{*})$
for 
$$\pi_{\kappa j(\kappa)}\circ\mathrm{pr}:
j(P(\kappa)\star\dot{Q}(\kappa,j(\kappa)))
\to P(\kappa)\star\dot{Q}(\kappa,j(\kappa)).$$

Let
$\bar{G}\subset P(j(\kappa))\star\dot{Q}( j(\kappa), j^{2}(\kappa))$ 
be
$V$-generic with
$(1_{j(P(\kappa))},\dot{s}^{*})\in\bar{G}$.
Work in $V[\bar{G}]$.
Let
$G=\pi_{\kappa j(\kappa)}\circ\mathrm{pr}(\bar{G})
\subset P(\kappa)\star \dot{Q}(\kappa, j(\kappa))$,
which is $V$-generic.
Then $j:V\to M$ extends to 
$j:V[G]\to M[\bar{G}]$.
Note that in $V$
$$\pi_{\kappa j(\kappa)}\times\mathrm{pr}_{\kappa}:
P(j(\kappa))\star\dot{Q}(j(\kappa), j^{2}(\kappa))\to
{P(\kappa)\star\dot{Q}(\kappa, j(\kappa))}*\dot{E}(j(\kappa), j^{2}(\kappa))$$
is a projection.
Thus we get a $V[G]$-generic 
$H\subset E(j(\kappa), j^{2}(\kappa))^{V[G]}$.
Note that 
$j[H]\in [M[\bar{G}]]^{j^{2}(\kappa)}\subset M[\bar{G}]$
because 
$|P(j(\kappa))\star\dot{Q}(j(\kappa), j^{2}(\kappa))|\le j^{2}(\kappa)$ 
in $V$.
Moreover $j[H]$ is a directed subset of
$E(j^{2}(\kappa),j^{3}(\kappa))^{M[\bar{G}]}$.

Let $q^{*}$ be the coordinatewise union of $j[H]$.
Then ${q}^{*}\in M[\bar{G}]$ is a map from
$$\textstyle\bigcup\{j(d):d \subset\mathsf{SR}\cap (j(\kappa),j^{2}(\kappa))
\text{ is Easton in }V[G]\},$$ 
which is an Easton subset of 
$\mathsf{SR}\cap(j^{2}(\kappa),j^{3}(\kappa))$.
To see that 
${q}^{*}\in E(j^{2}(\kappa),j^{3}(\kappa))^{M[\bar{G}]}$,
let $\xi\in\mathrm{dom}~{q}^{*}$.
Then
${q}^{*}(\xi)=\bigcup\{j(q)(\xi):q\in H\cap E(j(\kappa),\delta)^{V[G]}\}$
for some $\delta<j^{2}(\kappa)$.
Thus
${q}^{*}(\xi)\in\xi^{<j^{2}(\kappa)}$, as desired.

Let 
$\bar{H}\subset E( j^{2}(\kappa),j^{3}(\kappa))^{M[\bar{G}]}$
be $V[\bar{G}]$-generic with $q^{*}\in\bar{H}$.
Since $j[H]\subset \bar{H}$,
$j:V[G]\to M[\bar{G}]$ extends to 
$j:V[G*H]\to M[\bar{G}*\bar{H}]$
in $V[\bar{G}*\bar{H}]$.

Fix 
$f: j^{2}(\kappa)^{<\omega}\to j^{2}(\kappa)$ in $V[G*H]$.
Then in $M[\bar{G}*\bar{H}]$
there is
$x\in[j^{3}(\kappa))]^{j^{2}(\kappa)}$
closed under $j(f)$
such that 
$|x\cap j^{2}(\kappa)|=j(\kappa)$
and
$|x\cap j(\kappa)|=|\kappa|=\mu=j(\mu)$,
as witnessed by $j[j^{2}(\kappa)]$.
Thus in $V[G*H]$ there is 
$x\in[ j^{2}(\kappa)]^{j(\kappa)}$
closed under $f$ such that 
$|x\cap j(\kappa)|=\kappa$ and
$|x\cap\kappa|=\mu$, as desired.
\end{proof}

\end{document}